\documentclass[11pt]{amsart}

\usepackage{amsthm, mathrsfs,amssymb,amsmath,}
\usepackage{enumerate}

\makeatletter
\@namedef{subjclassname@2010}{%
\textup{2010} Mathematics Subject Classification}
\makeatother

\allowdisplaybreaks

\newtheorem{thm}{Theorem}[section]

\newtheorem{cor}[thm]{Corollary}

\theoremstyle{definition}

\theoremstyle{remark}

\def\Rn{{\mathbb R}^n}
\def\R2n{{\mathbb R}^{2n}}

\def\Rn{{\mathbb R}^n}

\def\R^2{{\mathbb R}^2}
\def\R2n{{\mathbb R}^{2n}}
\def\R{{\mathbb R}}

\title[Pseudo-differential Operators on the Abstract Heisenberg Group]{Hilbert-Schmidt and Trace Class Pseudo-differential Operators on the Abstract Heisenberg Group}
\author{Aparajita Dasgupta}
\address{Aparajita Dasgupta \endgraf Department of Mathematics \endgraf Indian Institute of Technology Delhi \endgraf Hauz Khas, New Delhi - 110 016, India.}
\email{adasgupta@maths.iitd.ac.in}
\author{Vishvesh Kumar} 
\address{Vishvesh Kumar \endgraf School of Mathematical Sciences \endgraf National Institute of Science Education and Research Bhubaneshwar, HBNI  \endgraf At/Po- Jatni, District- Khurda, Odisha- 7520250,  India.} 
\email{vishveshmishra@gmail.com}

\begin{document}
\begin{abstract} 
In this paper we introduce and study pseudo-differential operators with operator valued symbols on the abstract Heisenberg group $\mathbb{H}(G):=G \times \widehat{G} \times \mathbb{T},$ where $G$ a locally compact abelian group with its dual group $\widehat{G}$. We obtain a necessary and sufficient condition on symbols for which these operators are in the class of Hilbert-Schmidt operators. As a key step in proving this we derive a trace formula for the trace class $j$-Weyl transform, $j \in \mathbb{Z}^*$ with symbols in $L^{2}(G\times \widehat{G}).$ We go on to present a characterization of the trace class pseudo-differential operators on $\mathbb{H}(G)$. Finally, we also give a trace formula for these trace class operators.
\end{abstract}


\maketitle

\tableofcontents

\section{Introduction} The aim of this paper is to look at the pseudo-differential operators on the abstract Heisenberg group $\mathbb{H}(G)$, for locally compact abelian group $G.$ In 1964, A. Weil \cite{Wei} studied certain group of unitary operators associated with a locally compact abelian group in connection with the study of the celebrated work of Seigel on quadratic form. In \cite{Wei}, the author introduced this group where  he considered $G$ to be an ad$\grave{\text{e}}$le group or the additive group of vector space over a local field, which has applications in number theory.  Recently Radha, et al. \cite{Ra, Shr} studied Weyl multipliers and shift invariant spaces on the group $\mathbb{H}(G).$ It is a well-known fact from \cite{Wong} that pseudo-differential operators on $\Rn$ are based on the Plancherel formula for the Fourier transform on $\Rn$. The Plancherel formula gives rise to the Fourier inversion formula, which says that the identity
operator for $L^2(\Rn)$ can be expressed in terms of the Fourier transform on $\Rn$. The Fourier inversion formula, albeit useful in many situations, gives a perfect symmetry, namely, the identity operator.
 The classical pseudo-differential operator $T_{\sigma}$ associated to a symbol $\sigma,$ (a measurable function on $\mathbb{R}^{n}\times\mathbb{R}^n$) is defined by
\begin{equation}\label{clpdodef}
\left(T_{\sigma}\phi\right)(x)=(2\pi)^{-n/2}\int_{\mathbb{R}^n}e^{ix\xi}\sigma(x,\xi)\widehat{\phi}(\xi)d\xi,~~~x\in\mathbb{R}^{n},
\end{equation}
for all $\phi $ in the Schwartz space $S(\mathbb{R}^n)$ on $\mathbb{R}^n,$ provided the integral exists. The function $\widehat{\phi}$ in \eqref{clpdodef} is the Fourier transform of $\phi$ defined by
$$\widehat{\phi}(\xi)=(2\pi)^{-n/2}\int_{\mathbb{R}^n}e^{-ix\cdot\xi}\phi(x)dx,~~~\xi\in\mathbb{R}^n.$$
The  formation of a pseudo-differential operator as define in \eqref{clpdodef} is mainly based on the Fourier inversion formula given by,
$$\phi(x)=(2\pi)^{-n/2}\int_{\mathbb{R}^n}e^{ix\cdot\xi}\widehat{\phi}(\xi)d\xi,~~~x\in\mathbb{R}^n,$$ for all $\phi$ in $S(\mathbb{R}^n).$ By inserting a symbol, which is a suitable function on the phase space $\Rn\times\Rn,$ we break the symmetry and obtain the pseudo-differential operator. To extend pseudo-differential  operators to other settings, we first observe that $\Rn$ is a group and its dual is  also $\mathbb{R}^n.$ It is then natural to extend pseudo-differential operators to other groups which have explicit dual objects and Fourier inversion formulas. Recently, such works have been done for $\mathbb{S}^1,~~\mathbb{Z},$ finite abelian groups, locally compact abelian groups, affine groups, compact groups, homogeneous spaces of compact groups, Heisenberg group and on general locally compact type I groups \cite{Dasgupta, DW, Fis, Kumar, VW, Mol, Mol1, RuV} among others. \\
A basic result in the theory of pseudo-differential operators on $\Rn$ is that if $\sigma\in L^{2}(\Rn\times\Rn),$ then $T_{\sigma}$ can be extended to a bounded linear operator from $L^{2}(\Rn)$ into $L^{2}(\Rn).$ Moreover the resulting bounded linear operator is in Hilbert-Schimdt class as explained in \cite{Resi, Wo}. Recently in \cite{DW}, Dasgupta and Wong has obtained a necessary and sufficient conditions on the symbols for which the pseudo-differential operators on Heisenberg groups are in Hilbert-Schmidt class.  Motivated by this we wish to study the boundedness property of the pseudo-differential operators on $\mathbb{H}(G)$ and obtain conditions on the symbols for  which these operators are in Hilbert-Schimdt class.  In Section \ref{pre}, we recall the basics of the abstract Heisenberg group,$\mathbb{H}(G)$,  and define the pseudo-differential operators on the group $\mathbb{H}({G}),$ where $G$ is a locally compact abelian group. The $L^{2}$ boundedness property  of the pseudo-differential operators on $\mathbb{H}(G)$ is given in Section \ref{L2bdd}. In Section \ref{trWe} we obtained the trace formula of a trace class $j$-Weyl transform, $j \in \mathbb{Z}^*$ associated to the symbol in $L^{2}(G\times\widehat{G})$ , where $G$ is a locally compact abelian group and this result is the key step in obtaining the necessary and sufficient conditions for the pseudo-differential operators on the abstract Heisenberg group, $\mathbb{H}(G)$ to be in Hilbert-Schmidt class, which we proved in Section \ref{hilbdd}. In Section \ref{TrOp} trace class pseudo-differential operators on $\mathbb{H}(G)$ are given and a trace formula is obtained for them.

\section{Preliminaries}\label{pre}
Let $G$ be a locally compact abelian group such that the map $x \mapsto jx,$ $j \in \mathbb{Z}^{\ast}=\mathbb{Z}\setminus\{0\},$ is an automorphism of $G$. We denote the abstract Heisenberg group associated with $G$ by $\mathbb{H}(G):= G \times \widehat{G}\times \mathbb{T}$ with the group operation given by 
$$(x, \chi, \theta) (x', \chi', \theta)= (xx', \chi \chi', \theta \theta' \chi'(x)).$$
The unitary dual $\widehat{\mathbb{H}(G)} $ of $\mathbb{H}(G)$ can be identified with $\mathbb{Z}^{\ast},$  the set of non-zero integers, as follows: for any $j \in \mathbb{Z}^*,$ the irreducible representation of $\mathbb{H}(G)$ on $L^2(G)$ is given by 
$$(\rho_j(x, \chi, \theta) \varphi)(y)= \theta^j(\chi(y))^j \varphi(xy)\,\,\,\,\, \text{for}\,\, \varphi \in L^2(G).$$ It is well known from Stone-von Neumann theorem that every infinite dimensional irreducible unitary representation of $\mathbb{H}(G)$ is unitarily equivalent to $\rho_j,$ $j\in \mathbb{Z}^*.$ For further details we refer Folland \cite{Fo1, Fo2}.

For each $j \in \mathbb{Z}^*,$ the Fourier transform, $\widehat{f}$ of $f \in L^1(\mathbb{H}(G))$ at $j$  is an operator $\widehat{f}(j)$ on $L^{2}(G)$ given  by 
$$\widehat{f}(j) \psi:= \int_{G \times \widehat{G} \times \mathbb{T}} f(x, \chi, \theta) \, \rho_j(x, \chi, \theta)\psi \, d\mu_G\, d\mu_{\widehat{G}}\, d\mu_{\mathbb{T}},$$ where $d\mu_G\, d\mu_{\widehat{G}}\, d\mu_{\mathbb{T}}$ is a product of the Haar measure on $G,$ $\widehat{G}$ and $\mathbb{T}.$ The above integral is a Bochner  integral taking values in the Hilbert space $L^{2}(G).$ This operator $\widehat{f}(j)$ is a bounded operator on $L^{2}(G)$ with $||\widehat{f}(j)||_{\mathcal{B}(L^2(G))}\leq ||f||_{L^{1}(G)}.$ The inverse Fourier transform, $f^{j},$ $j\in\mathbb{Z}^{\ast},$ of $f\in L^{1}(\mathbb{H}(G))$ in the $\theta$ variable is given by
\begin{equation}\label{InFT}
f^{j}(x,y)=\int_{\mathbb{T}}f(x,\gamma,\theta)\theta^{j}d\mu_{\mathbb{T}}(\theta). 
\end{equation} 

Then $f^{j}\in L^{1}(G\times\widehat{G})$ and using \eqref{InFT}, $\widehat{f}$ can be written as
\begin{equation}\label{FT}
\widehat{f}(j)\phi=\int_{G\times \widehat{G}}f^{j}(x,\gamma)\rho_{j}(x,\gamma,1)\phi  d\mu_G\, d\mu_{\widehat{G}}, ~~~\phi\in L^{2}(G).
\end{equation}
Further the Fourier transform $\widehat{f}$ of $f\in L^{2}(G)$ satisfies 
$$||\widehat{f}(j)||_{\mathcal{B}_{2}(L^2(G))}=(C_{j,G})^{-1}||f^{j}||_{L^{2}(G\times\widehat{G})}^{2},$$ where  $C_{j,G}$ is a constant depending on $j$ and $G$ such that $C_{j,G} d\mu_G(j^{-1}x)=d\mu_G(x),$ for $j \in \mathbb{Z}^*$ and $C_{0,G}$ is taken be to zero. \\
Since the group is $\mathbb{H}(G)=G\times \widehat{G}\times\mathbb{T}$ so to obtain the Plancheral theorem we need to consider the one-dimensional representations of $\mathbb{H}(G)$ in addition to the infinite dimensional irreducible unitary representations. Then the Plancherel formula for $f \in L^2(\mathbb{H}(G))$ is given by
\begin{equation}\label{Plth}
\|f\|^2_{L^2(\mathbb{H}(G))} = \|\widehat{f}\|^2_{\ell^2(\mathbb{Z}^*, S_2, C_{j, G})}+ \int_{\widehat{G}} \int_G |\widehat{(f^0)}(\gamma, y)|^2 \, d\mu_G(y)\, d\mu_{\widehat{G}}(\gamma),
\end{equation}
where $f_0(x, \chi)= \int_{\mathbb{T}} f(x, \chi, \theta)\, d\mu_{\mathbb{T}}(\theta)$. However for the case when $f-f^{0},$ $f\in L^{2}(\mathbb{H}(G))$ has mean value zero in the central variable, the Plancheral formula reduces to the following,
\begin{equation}\label{nePl}
||\widehat{f}||_{l^{2}(\mathbb{Z}\setminus\{0\}, \mathcal{B}(L^{2}(G); C_{j,G})}=||f-f^{0}||_{L^{2}(\mathbb{H}(G))}.
\end{equation}
 In other words, we can say there is an isometry from $L^{2}(\mathbb{H}(G))$ into $l^{2}(\mathbb{Z}\setminus\{0\}, \mathcal{B}(L^{2}(G); C_{j,G}),$ where $l^{2}(\mathbb{Z}\setminus\{0\}, \mathcal{B}(L^{2}(G); C_{j,G})$ denotes the space of all sequences indexed in $\mathbb{Z}^{\ast},$ taking values in $\mathcal{B}(L^{2}(G)$ and square summable with weight $C_{j,G}.$ See \cite{Ra,Wei} for more details.   

The following Fourier Inversion formula for the Fourier transform on the abstract Heisenberg group is the starting point for the analysis of the pseudo-differential operators on the abstract Heisenberg group.\\
For a function with compact support, $f\in \mathcal{C}_{c}(\mathbb{H}(G)),$  the Fourier Inversion formula is given by,
\begin{equation}\label{IFT}
f(x,\chi,\theta)=\sum_{j \in \mathbb{Z}^*} tr( \rho_j^*(x, \chi, \theta) \widehat{f}(j)),
\end{equation}

where $\rho_j^{\ast}(x, \chi, \theta)$ is the adjoint of $\rho_j(x, \chi, \theta).$\\
Now let $\mathbb{B}(L^{2}(G))$ be the $C^{\ast}$-algebra of bounded linear operators on $L^{2}(G).$ Then we call the mapping $\sigma: \mathbb{H}(G)\times \mathbb{Z}^{\ast}\rightarrow B(L^{2}(G))$ an operator valued symbol or simply a symbol. 

We define a pseudo-differential operator $T_\sigma$ associated with the operator-valued symbol $\sigma: \mathbb{H}(G) \times \mathbb{Z}^* \rightarrow \mathcal{B}(L^2(G))$ as 

$$(T_\sigma f)(x, \chi, \theta):= \sum_{j \in \mathbb{Z}^*} tr( \rho_j^*(x, \chi, \theta) \sigma(x, \chi, \theta, j) \widehat{f}(j)) \,\,\,\,\, \forall f \in C_c(\mathbb{H}(G)).$$
 
\section{$L^2$-boundedness } \label{L2bdd}
In this section we prove $L^2$- boundedness of pseudo-differential operators on $\mathbb{H}(G).$ We also show that if two symbols with some conditions give arise to same pseudo-differential operator then symbol must be same. Before stating our results of this section we would like to recall some notation from operator theory. 

Let $X$ be a complex and separable Hilbert space with the inner product denoted by $(\cdot, \cdot)$ and let  $\{\psi_{k}, k=1,2,..\}$ be an orthonormal basis for $X.$ 
 
 An operator $T \in \mathcal{B}(X)$ is a Hilbert-Schmidt operator if for any (hence all) orthonormal basis $\{\psi_k\}_{k=1}^\infty$ of $X$ we have $\sum_k \|T\psi_k\|_{X}< \infty.$ The set of Hilbert-Schmidt operators, denoted by $S_2(X),$ is two sided ideal of $\mathcal{B}(X).$ At times we also denote $S_2(X)$ only by $S_2.$ The Hilbert-Schmidt norm on $S_2$ is given by $$\|T\|_{HS}= \left(\sum_{k=1}^\infty \|T\psi_k\|_{X}^2 \right)^\frac{1}{2}.$$ 
 
Now, we are ready to state the following theorem on $L^2$-boundedness of pseudo-differential operators on $\mathbb{H}(G). $

\begin{thm}
	Let $\sigma: \mathbb{H}(G) \times \mathbb{Z}^* \rightarrow S_2$ be a operator-valued symbol such that 
	$$\sum_{j \in \mathbb{Z}^*} \int_{G \times \widehat{G} \times \mathbb{T}} C_{j, G}^{-1}\, \|\sigma(x, \chi, \theta, j)\|^2_{S_2}\, d\mu_G\, d\mu_{\widehat{G}} \, d\mu_{\mathbb{T}} < \infty.$$ 
	Then the pseudo-differential operator $T_\sigma:L^2(\mathbb{H}(G)) \rightarrow L^2(\mathbb{H}(G))$ is a bounded operator.
\end{thm}

\begin{proof}   Using Minkowski integral inequality and the estimate $\|\widehat{f}\|_{\ell^2(\mathbb{Z}^*, S_2, C_{j, G})} \leq \|f\|_{L^2(\mathbb{H}(G))}$ we have, for $f \in L^2(\mathbb{H}(G)),$
	\begin{align*} 
	\|T_\sigma f\|_{L^2(\mathbb{H}(G))}& = \left( \int_{G\times \widehat{G} \times \mathbb{T}} |(T_\sigma f)(x, \chi, \theta)|^2 d\mu_G(x) \, d\mu_{\widehat{G}}(\chi) \,d\mu_{\mathbb{T}}(\theta)\right)^{\frac{1}{2}} \\ & = \left( \int_{G\times \widehat{G} \times \mathbb{T}} \left| \sum_{j \in \mathbb{Z}^*} tr( \rho_j^*(x, \chi, \theta) \sigma(x, \chi, \theta, j) \widehat{f}(j))\right|^2 d\mu_G(x) \, d\mu_{\widehat{G}}(\chi) \,d\mu_{\mathbb{T}}(\theta)\right)^{\frac{1}{2}} \\ & \leq \sum_{j \in \mathbb{Z}^*} \left( \int_{G\times \widehat{G} \times \mathbb{T}} \left|  tr( \rho_j^*(x, \chi, \theta) \sigma(x, \chi, \theta, j) \widehat{f}(j))\right|^2 d\mu_G(x) \, d\mu_{\widehat{G}}(\chi) \,d\mu_{\mathbb{T}}(\theta)\right)^{\frac{1}{2}} \\ & \leq \sum_{j \in \mathbb{Z}^*} \left( \int_{G\times \widehat{G} \times \mathbb{T}} \|\sigma(x, \chi, \theta, j)\|_{S_2}^2\,\, \|\widehat{f}(j))\|_{S_2}^2  d\mu_G(x) \, d\mu_{\widehat{G}}(\chi) \,d\mu_{\mathbb{T}}(\theta)\right)^{\frac{1}{2}} \\&= \sum_{j \in \mathbb{Z}^*} \|\widehat{f}(j))\|_{S_2} \left( \int_{G\times \widehat{G} \times \mathbb{T}} \|\sigma(x, \chi, \theta, j)\|_{S_2}^2\,\,  d\mu_G(x) \, d\mu_{\widehat{G}}(\chi) \,d\mu_{\mathbb{T}}(\theta)\right)^{\frac{1}{2}} \\& \leq \left( \sum_{j \in \mathbb{Z}^*} C_{j,G} \|\widehat{f}(j))\|_{S_2}^2 \right)^{\frac{1}{2}} \times \\ & \left( \sum_{j \in \mathbb{Z}^*}   \int_{G\times \widehat{G} \times \mathbb{T}} C_{j,G}^{-1} \|\sigma(x, \chi, \theta, j)\|_{S_2}^2\,\,  d\mu_G(x) \, d\mu_{\widehat{G}}(\chi) \,d\mu_{\mathbb{T}}(\theta)\right)^{\frac{1}{2}} \\& = \|\widehat{f}\|_{\ell^2(\mathbb{Z}^*, S_2, C_{j, G})} \left( \sum_{j \in \mathbb{Z}^*}   \int_{G\times \widehat{G} \times \mathbb{T}} C_{j,G}^{-1} \|\sigma(x, \chi, \theta, j)\|_{S_2}^2\,\,  d\mu_G(x) \, d\mu_{\widehat{G}}(\chi) \,d\mu_{\mathbb{T}}(\theta)\right)^{\frac{1}{2}} \\& \leq \|f\|_{L^2(\mathbb{H}(G))} \left( \sum_{j \in \mathbb{Z}^*}   \int_{G\times \widehat{G} \times \mathbb{T}} C_{j,G}^{-1} \|\sigma(x, \chi, \theta, j)\|_{S_2}^2\,\,  d\mu_G(x) \, d\mu_{\widehat{G}}(\chi) \,d\mu_{\mathbb{T}}(\theta)\right)^{\frac{1}{2}}.   
	\end{align*}
	Therefore, 
	  $$\|T_\sigma\|_{\mathcal{B}(L^2(\mathbb{H}(G)))} \leq \left( \sum_{j \in \mathbb{Z}^*}   \int_{G\times \widehat{G} \times \mathbb{T}} C_{j,G}^{-1} \|\sigma(x, \chi, \theta, j)\|_{S_2}^2\,\,  d\mu_G(x) \, d\mu_{\widehat{G}}(\chi) \,d\mu_{\mathbb{T}}(\theta)\right)^{\frac{1}{2}} <\infty.$$
	  Hence, $T_\sigma$ is a bounded operator on $L^2(\mathbb{H}(G)).$
	\end{proof}
	In the next theorem we prove that two symbols giving the same pseudo-differential operators are equal.
\begin{thm} \label{thm3.2}
	Let $\sigma : \mathbb{H}(G) \times \mathbb{Z}^* \rightarrow S_2$ such that 
	\begin{equation} \label{1}
	\sum_{j \in \mathbb{Z}^*} \int_{G \times \widehat{G} \times \mathbb{T}} C_{j, G}^{-1}\, \|\sigma(x, \chi, \theta, j)\|^2_{S_2}\, d\mu_G\, d\mu_{\widehat{G}} \, d\mu_{\mathbb{T}} < \infty.
	\end{equation} 
	Furthermore, suppose that 
	\begin{equation} \label{2}
	\sum_{j \in \mathbb{Z}^*}  \|\sigma(x, \chi, \theta, j)\|_{S_2}<\infty, \,\,\,\,\,\,\,\,\, \sup_{x, \chi, \theta, j} \|\sigma(x, \chi, \theta, j)\|_{S_2} <\infty
	\end{equation}  and the mapping $(x, \chi, \theta, j) \mapsto \rho_j^*(x, \chi, \theta) \sigma(x, \chi, \theta, j)$ from $\mathbb{H}(G) \times \mathbb{Z}^*$ into $S_2$ is weakly continuous. Then $T_\sigma f= 0\,\,\, \forall f \in L^2(\mathbb{H}(G))$ if and only if $\sigma(x, \chi, \theta, j)=0$ for almost all $(x, \chi, \theta, j) \in \mathbb{H}(G) \times \mathbb{Z}^*.$
	\end{thm}
\begin{proof} If $\sigma(x, \chi, \theta, j)=0$ for almost all $(x, \chi, \theta, j) \in \mathbb{H}(G) \times \mathbb{Z}^*$ then it is obvious that $T_\sigma f=0$ for all $f \in f \in \mathbb{H}(G).$ 
	
	It remains to show that if $T_\sigma f= 0\,\,\, \forall f \in L^2(\mathbb{H}(G))$ then $\sigma(x, \chi, \theta, j)=0$ for almost all $(x, \chi, \theta, j) \in \mathbb{H}(G) \times \mathbb{Z}^*.$
	
	For $(x, \chi, \theta) \in \mathbb{H}(G),$ define a function $f_{x, \chi, \theta} \in L^2(\mathbb{H}(G))$ as 
	$$\widehat{f}_{x, \chi, \theta}(j)= \sigma(x, \chi, \theta, j)^* \rho_j(x, \chi, \theta)\,\,\,\, j \in \mathbb{Z}^*.$$ 
	
	For all $(x', \chi', \theta') \in \mathbb{H}(G),$ 
	$$(T_\sigma f_{x, \chi, \theta})(x', \chi', \theta')= \sum_{j \in \mathbb{Z}^*} tr(\rho_j^*(x', \chi', \theta') \sigma(x', \chi', \theta', j ) \sigma(x, \chi, \theta, j)^* \rho_j(x, \chi, \theta)).$$
	
	Now, let $(x_0, \chi_0, \theta_0) \in \mathbb{H}(G).$ Then, using the weak continuity of the mapping $$(x, \chi, \theta, j) \mapsto \rho_j^*(x, \chi, \theta) \sigma(x, \chi, \theta, j)$$ we get that 
	 \begin{align*}
	 &tr(\rho_j^*(x', \chi', \theta') \sigma(x', \chi', \theta', j ) \sigma(x, \chi, \theta, j)^* \rho_j(x, \chi, \theta))\\&  \rightarrow tr(\rho_j^*(x_0, \chi_0, \theta_0) \sigma(x_0, \chi_0, \theta_0, j ) \sigma(x, \chi, \theta, j)^* \rho_j(x, \chi, \theta))
	 \end{align*} 
	 as $(x', \chi', \theta', ) \rightarrow (x_0, \chi_0, \theta_0)$ in $\mathbb{H}(G).$ Moreover, by using  $\sup_{x, \chi, \theta, j} \|\sigma(x, \chi, \theta, j)\|_{S_2} <\infty$ there exists a constant $K>0$ such that for all $(x', \chi', \theta',j) \in \mathbb{H}(G) \times \mathbb{Z}^*,$
	 
	  $$|tr(\rho_j^*(x', \chi', \theta') \sigma(x', \chi', \theta', j ) \sigma(x, \chi, \theta, j)^* \rho_j(x, \chi, \theta))| \leq K \|\sigma(x, \chi, \theta, j)\|_{S_2}.$$
	  Now, using the assumption that $	\sum_{j \in \mathbb{Z}^*}  \|\sigma(x, \chi, \theta, j)\|_{S_2}<\infty,$ we get, an application of  Lebesgue dominated convergence theorem, that  
	 \begin{align*}
	&\sum_{j \in \mathbb{Z}^*}tr(\rho_j^*(x', \chi', \theta') \sigma(x', \chi', \theta', j ) \sigma(x, \chi, \theta, j)^* \rho_j(x, \chi, \theta))\\&  \rightarrow \sum_{j \in \mathbb{Z}^*} tr(\rho_j^*(x_0, \chi_0, \theta_0) \sigma(x_0, \chi_0, \theta_0, j ) \sigma(x, \chi, \theta, j)^* \rho_j(x, \chi, \theta))
	\end{align*} 
	as $(x', \chi', \theta', ) \rightarrow (x_0, \chi_0, \theta_0)$ in $\mathbb{H}(G).$ Therefore, $T_\sigma f_{x, \chi, \theta}$ is continuous. 
	
	Next, by letting $(x_0, \chi_0, \theta_0)= (x, \chi, \theta)$ we get
	\begin{align*}
	T_\sigma f_{x, \chi, \theta} (x, \chi, \theta)&= \sum_{j \in \mathbb{Z}^*} tr(\rho_j^*(x, \chi, \theta) \sigma(x, \chi, \theta, j ) \sigma(x, \chi, \theta, j)^* \rho_j(x, \chi, \theta)) \\& =\sum_{j \in \mathbb{Z}^*} tr( \sigma(x, \chi, \theta, j ) \sigma(x, \chi, \theta, j) ^*) \\&= \sum_{j \in \mathbb{Z}^*} \|\sigma(x, \chi, \theta, j)\|_{S_2}^2 =0
	\end{align*}
	So, $\|\sigma(x, \chi, \theta, j)\|_{S_2}=0$ for almost every $j \in \mathbb{Z}^*$ and hence $\sigma(x, \chi, \theta, j)=0$ for almost all $(x, \chi, \theta, j) \in \mathbb{H}(G) \times \mathbb{Z}^*. $ 
\end{proof}

\section{Trace of Weyl transform on a locally compact Abelian group}\label{trWe} 
 Let $X$ be a complex and separable Hilbert space in which the inner product is denoted by $(,)$ and let $A: X\rightarrow X$ be a compact operator. If we denote  by $A^{\ast}: X\rightarrow X$ the adjoint of $A:X\rightarrow X$ then the linear operator $(A^{\ast}A)^{1/2}:X\rightarrow X$ is positive and compact. Let $\{\psi_{k}, k=1,2,..\}$ be an orthonormal basis for $X$ consisting of eigenvalues of $(A^{\ast}A)^{1/2}:X\rightarrow X$ and let $s_{k}(A)$ be the eigenvalue corresponding to the eigenvector $\psi_{k}, k=1,2,3,...$. Then $s_k(A)$ $k=1,2,3,...$, are the singular values of $A:X\rightarrow X$. If $$\sum\limits_{k=1}^{\infty}s_{k}(A)<\infty,$$ then the linear operator $A:X\rightarrow X$ is said to be in the trace class $S_{1}.$ It can be shown that $S_1$ is a Banach space in  which the norm $||\cdot||_{S_1}$ is given by
 $$||A||_{S_1}=\sum\limits_{k=1}^{\infty}s_{k}(A),~~~A\in S_1.$$
 Let $A: X\rightarrow X$ be a linear operator in $S_1$ and let $\{\phi_{k}: k=1,2,3,...\}$ be any orthonormal basis for $X$.  Then from \cite{Resi}, the series $\sum\limits_{k=1}^{\infty}(A\phi_k,\phi_k)$ is absolutely convergent and the sum is independent of the choice of the orthonormal basis $\{\phi_{k}:k=1,2,3,...\}.$ Thus the trace of any linear operator $A: X\rightarrow X$ in $S_1$ is defined by
 $$\text{tr}(A)=\sum\limits_{k=1}^{\infty}(A\phi_k,\phi_k),$$ where $\{\phi_{k}, k=1,2,3...\}$ is any orthonormal basis of $X$. \\

 The following well-known theorem describe a relation between trace class operator and Hilbert-Schmidt operators. 
 \begin{thm}\label{trace} Let $T \in \mathcal{B}(\mathcal{H}).$ Then $T$ is a trace class operator if and only if there exist two Hilbert-Schmidt operators $U$ and $V$ on $\mathcal{H}$ such that $T=UV.$
 \end{thm}

 In this section we have obtained the trace formula for the trace class Weyl transform associated to a symbol in $L^{2}(G\times \widehat{G}),$ where $G$ is a locally compact abelian group. \\
  Here first we recall that the Schr\"{o}dinger representation $\rho_j,$ $j \in \mathbb{Z}^*$ of the abstract Heisenberg group $\mathbb{H}(G)$  on $L^2(G)$ is defined by 
     $$(\rho_j(x, \chi, \theta) \varphi) (y)= \theta^j \chi(y)^j \varphi(x y)\,\,\,\,\,\, \varphi \in L^2(G).$$
Using $\rho_j,$ we define the {\it $j$-Weyl transform } $W^j: \sigma \mapsto W^j_\sigma$ from $ C_c(G \times \widehat{G})$ into $ \mathcal{B}(L^2(G))$ as 
$$W_\sigma^j (\varphi)(y)= \int_{G\times \widehat{G}} \sigma(x, \chi) (\rho_j(x, \chi, 1) \varphi)(y) \,d\mu_{\widehat{G}}(\chi) \, d\mu_G(x).$$ 

Which is further can be written as   
$$W^j_\sigma (\varphi)(y)= \int_{G\times \widehat{G}} \sigma(x, \chi) \chi(y)^j \varphi(x y) \,d\mu_{\widehat{G}}(\chi) \, d\mu_G(x)= \int_G K^j_\sigma(x,y) \, d\mu_G(x), $$
where 
\begin{equation} \label{sim}
K^j_\sigma(x,y)= \int_{\widehat{G}} \sigma(x y^{-1}, \chi) \chi(y)^j \, d\mu_{\widehat{G}}(\chi). 
\end{equation} 

Therefore, $W_\sigma^j$ is an integral operator with kernel $K^j_\sigma.$ Note that for $\sigma \in L^2(G \times \widehat{G}),$ the $j$-Weyl transform $W_\sigma^j$ is a Hilbert Schmidt operator. In fact, 
\begin{equation} \label{eq4}
\|W_\sigma^j\|_{S_2}= \|K_\sigma^j\|_{L^2(G \times \widehat{G})} =(C_{j,G})^{-1} \| \sigma \|_{L^2(G \times \widehat{G})}
\end{equation}
 and more generally, $\langle W_f^j, W_g^j \rangle_{S_2}= (C_{j,G})^{-1} \langle f, g \rangle_{L^2(G \times \widehat{G})}.$

For two function $f$ and $g$ in $L^2(G \times \widehat{G}),$ the { \it $j$-twisted convolution} $f \times_j g,$ $j \in \mathbb{Z}^*$ of $f$ and $g$ is defined by 

$$ f \times_j g (x, \chi) = \int_{G} \int_{\widehat{G}} f(x', \chi') g(xx'^{-1}, \chi \chi'^{-1}) \overline{\chi(x')}^j \, d\mu_G(x')\, d\mu_{\widehat{G}}(\chi').$$ 

Note for $j=1,$ $W^j_\sigma$ turns out to be the well known Weyl transform $W_\sigma$  and $j$- twisted convolution is nothing but twisted convolution studied, see \cite{Ra,Shr}.
In the same line of \cite{Shr} it can be  shown that $(L^2(G \times \widehat{G}), \times_j)$ is a Banach algebra. Further, the $j$-Weyl transform is a Banach algebra isomorphism from $L^2(G \times \widehat{G})$ onto the space of all Hilbert-Schmidt operators on $L^2(G)$ denoted by $S_2(L^2(G))$ or $S_2.$  
Therefore, for any $A \in S_2(L^2(G)),$ there exist a unique $\sigma \in L^2(G \times \widehat{G})$ such that $A= W_\sigma^j.$ Also, $W_\sigma^j W_\tau^j= W_{\sigma \times_j \tau}^j.$ 

Denote the subset of all $\lambda \in L^2(G \times \widehat{G}) $ such that there exist functions $\sigma, \tau \in L^2(G \times \widehat{G})$ such that $\lambda = \sigma \times_j \tau$ by $W_j.$

Now we present the following theorem on the characterization of trace class $j$-Weyl transform.  

\begin{thm} \label{4.1}
	Let $W_\sigma^j$ be the $j$-Weyl transform on $G$ associated with symbol $\sigma \in L^2(G \times \widehat{G}).$ Then $W_\sigma^j$ is a trace class operator if and only if $\sigma \in W_j.$ 
\end{thm}

\begin{proof}
	Let $\sigma \in W_j.$ Then there exist $\lambda$ and $\tau$ in $L^2(G \times \widehat{G})$ such that $\sigma = \lambda \times_j \tau.$ So, 
	$$W_\sigma^j= W_{\lambda \times_j \tau}^j = W_\lambda^j W_\tau^j.$$ 
	
	Since $\tau$ and $\lambda$ are in $L^2(G \times \widehat{G}).$ It follows that $W_\lambda$ and $W_\tau$ are Hilbert Schmidt operators. Therefore, $W_\sigma$ is a trace class operator being product of two Hilbert Schmidt operators. 
	
	Conversely, suppose that $W_\sigma^j$ is a trace class operator. Then, it follows that $W_\sigma= A B$ for some Hilbert Schmidt operators  $A$ and $B$ on $L^2(G).$ 
	Since $j$-Weyl transform is an algebra isomorphism from $L^2(G \times \widehat{G})$ onto $S_2,$ it follows that there exist $\lambda, \tau \in L^2(G \times \widehat{G}) $ such that $A= W_\lambda^j$ and $B= W_\tau^j$ and therefore, $W_\sigma^j= W_\lambda^j W_\tau^j= W_{\lambda \times_j \tau}^j.$ Hence, $\sigma \in W_j.$ 
\end{proof}

The following corollary of Theorem \ref{4.1} is immediate once we recall that $\sigma_j \mapsto W_\sigma^j$ is a Banach algebra isomorphism from $(L^2(G \times \widehat{G}), \times_j )$ onto $S_2(L^2(G)).$
\begin{cor} \label{4.2}
	The $W_j$ is a subspace of $L^2(G \times \widehat{G}).$
\end{cor}
 
 In addition to this corollary, we have the following result the space $W_j.$ 

\begin{thm} $W_j$ is a dense subspace of $L^2(G \times \widehat{G}).$
\end{thm}
\begin{proof}
   In view of Corollary \ref{4.2},	we only need to prove that $W_j$ is dense in $L^2(G\times \widehat{G}).$ Let $F$ be the set of all functions $\sigma$ on $G\times \widehat{G}$  such that $W_\sigma^j$ is a finite rank operator on $L^2(G)$. Since every element in $S_2(L^2(G))$ is the limit in $S_2(L^2(G))$ of a sequence of finite rank operators on $L^2(G)$ and $W_\sigma^j$ is in $S_2(L^2(G))$ if and only if $\sigma \in L^2(G \times \widehat{G}),$ it follows that $F$ is a dense subspace of $L^2(G\times \widehat{G})$. Obviously, $F$ is a subspace of $W_j$. Therefore $W_j$ is dense in $L^2(G\times \widehat{G}).$
\end{proof}

Now we calculate the trace of trace of $j$-Weyl transform $W_\sigma^j$ associated with symbol $\sigma.$

\begin{thm}
     Let $\sigma  \in L^2(G \times \widehat{G})$ such that the $j$-Weyl transform $W_\sigma^j$ is a trace class operator. Then 
    $$ tr(W_\sigma^j)= \int_G \int_{\widehat{G}} \sigma(x, \chi)\, d\mu_{\widehat{G}}(\chi) \,d\mu_G(x).$$  
\end{thm}

\begin{proof}
	Since $\sigma \in L^2(G \times \widehat{G}),$ it follows that $W_\sigma^j$ is a Hilbert Schmidt operator with kernel $K_\sigma^j(x,y)$ given by 
	$$ K_\sigma^j(x,y)= \int_{\widehat{G}} \sigma(x y^{-1}, \chi) \chi(y)^j \, d\mu_{\widehat{G}}(\chi).$$
	
	Now, as by assumption $W_\sigma^j$ is a trace class operator so $\int_G K_\sigma^j(x, x) d\mu_G(x)$ exists and 
	\begin{align*}
	tr(W_\sigma^j) &= \int_G K_\sigma^j(x,x) \, d\mu_G(x) \\ &= \int_G \int_{\widehat{G}} \sigma(e, \chi) \, \chi(x)^j\, d\mu_{\widehat{G}}(\chi)\, d\mu_G(x) \\ &= \int_{G} \int_{\widehat{G}} \sigma(x^{-1}, \chi) \,\chi(e)^j \,d\mu_{\widehat{G}}(\chi)\, d\mu_G(x) \\&= \int_{G} \int_{\widehat{G}} \sigma(x, \chi)\, \chi(e)^j \,d\mu_{\widehat{G}}(\chi)\, d\mu_G(x)
	\end{align*}
	 where $e$ denotes the identity element of $G.$ Since $\chi(e)=1,$ we get 
	 
	 $$ tr(W_\sigma^j)= \int_{G} \int_{\widehat{G}} \sigma(x, \chi)\,d\mu_{\widehat{G}}(\chi)\, d\mu_G(x).$$ 
	\end{proof}

We will give another formula for the trace of a trace class Weyl transform which will be useful for our study in next section. 

\begin{thm} \label{4.5}
 Let $\sigma = \lambda \times_j \tau $ for some $\lambda, \tau \in L^2(G \times \widehat{G})$ such that the $j$-Weyl transform $W_\sigma^j$ is a trace class operator. Then 
 \begin{align} \label{eq4.5}
 tr(W_\sigma^j) = (C_{j,G})^{-1}\int_{G \times \widehat{G}} \tau(x, \chi)  \, \lambda(x^{-1}, \chi)\, \overline{\chi(x)}\, d\mu_G(x) \, d\mu_{\widehat{G}}(\chi).
 \end{align}  
\end{thm}
	
\begin{proof} First note that
	\begin{equation}
	W_\sigma^j=  W_\lambda^j W_\tau^j = W_{\lambda\times_j \tau}^j
	\end{equation}
	
	Since $W_\sigma^j$ is a trace class operator and hence a Hilbert Schmidt operator on $L^2(G)$. Let $\{\varphi_k: k \in \mathbb{N} \}$ is an orthonormal basis for $L^2(G).$ Then, by using the fact that $(W^j_\lambda)^*= W_{\tilde{\lambda}},$ where $\tilde{\lambda}(x, \chi)=\chi(x)\, \overline{\lambda(x^{-1}, \chi)},$ we have 
	
	\begin{align*}
	tr(W_\sigma^j) &= \sum_{k \in \mathbb{N}} \langle W_\sigma^j \varphi_k, \varphi_k \rangle = \sum_{k \in \mathbb{N}} \langle W_\lambda^j W_\tau^j \varphi_k, \varphi_k \rangle \\ &=  \sum_{k \in \mathbb{N}} \langle  W_\tau^j \varphi_k, (W^j_\lambda) ^* \varphi_k \rangle = \sum_{k \in \mathbb{N}} \langle  W_\tau^j \varphi_k, W_{\tilde{\lambda}}^j \varphi_k \rangle = \langle W_\tau^j, W_{\tilde{\lambda}}^j \rangle_{S_2}.
	\end{align*} 
	By Using the relation $\langle W_f^j, W_g^j \rangle_{S_2}= (C_{j,G})^{-1} \langle f, g \rangle_{L^2(G \times \widehat{G})},$ we get 
	
\begin{align}
tr(W_\sigma^j) &= (C_{j,G})^{-1} \langle \tau, \tilde{\lambda} \rangle_{L^2(G \times \widehat{G})} \\&= (C_{j,G})^{-1} \int_{G \times \widehat{G}} \overline{\chi(x)}\, \tau(x, \chi) \, \lambda(x^{-1}, \chi)\, d\mu_G(x) \, d\mu_{\widehat{G}}(\chi). 
\end{align} 
	\end{proof}

\section{Hilbert-Schmidt pseudo-differential operators on the abstract Heisenberg groups }\label{hilbdd} 

In this section, we characterize the Hilbert-Schmidt pseudo-differential operators in terms of their corresponding symbols. We begin this section with following observation.

Let $f \in L^2(\mathbb{H}(G)).$ For $j \in \mathbb{Z}^*,$  we know that $f^j$ is defined as  

$$f^j(x, \chi)= \int_{\mathbb{T}} f(x, \chi, \theta) \, \theta^j\, d\theta, ~~(x,\chi,\theta)\in G\times\widehat{G}\times \mathbb{T}.$$

Note that $f^j$ is the inverse Fourier transform of $f$ in $\theta$ variable or Fourier transform of $f$ with respect to the center of $\mathbb{H}(G).$ Therefore, it is convenient to write $f^j$ is the following form: 

$$f^j(x, \chi)= (\mathcal{F}_c^{-1}f)(x, \chi, j)= (\mathcal{F}_c f)(x, \chi, -j),$$

where $\mathcal{F}_c$ denote the Fourier transform with respect to center of $\mathbb{H}(G).$   

Before stating our main theorem of this section, we would like to note $\widehat{f}(j) \varphi= W_{f^j}^j(\varphi)$ for all $\varphi \in L^2(\mathbb{H}(G)).$ In fact,
 \begin{align*}
 \widehat{f}(j) \varphi(y) &= \int_{G \times \widehat{G} \times \mathbb{T}} f(x, \chi, \theta) \,(\rho_j(x, \chi, \theta) \varphi)(y) d\mu_G(x)\, d \mu_{\widehat{G}}(\chi)\, d\mu_{\mathbb{T}}(\theta) \\&=  \int_{G \times \widehat{G} \times \mathbb{T}}  f(x, \chi, \theta) \, \theta^j (\rho_j(x, \chi, 1)\varphi)(y)   d\mu_G(x)\, d\mu_{\widehat{G}}(\chi)\, d\mu_{\mathbb{T}}(\theta) \\&= \int_{G\times \widehat{G}} f^j(x, \chi)\, \chi(x)^j \varphi(xy)  d\mu_G(x)\, d\mu_{\widehat{G}}(\chi) \\&= (W_{f^j}^j\varphi)(y).
 \end{align*}

The following theorem is the main theorem of this section which characterizes Hilbert-Schmidt pseudo-differential operators on $L^2(\mathbb{H}(G)).$ 

\begin{thm} \label{thm5.1}
	Let $\sigma: \mathbb{H}(G) \times \mathbb{Z}^* \rightarrow S_2$ be a symbol such that the hypothesis of Theorem \ref{thm3.2} is satisfied. Then the corresponding pseudo-differential operator $T_\sigma:L^2(\mathbb{H}(G)) \rightarrow \mathbb{H}(G)$ is Hilbert-Schmidt operator if and only if 
	$$\sigma(x, \chi, \theta, j)= C_{j,G} \, \rho_j(x, \chi, \theta) W^j_{\alpha(x, \chi, \theta)^{-j}}\,\,\,\,\, (x, \chi, \theta) \in \mathbb{H}(G),\, j \in \mathbb{Z}^*,$$
	where $\alpha: \mathbb{H}(G) \rightarrow L^2(\mathbb{H}(G))$ is a weakly continuous mapping for which 
	$$\int_{G \times \widehat{G} \times \mathbb{T}}\|\alpha(x, \chi, \theta)\|_{L^2(\mathbb{H}(G))}^2 d\mu_G(x)\, d\mu_{\widehat{G}}(\chi)\, d\mu_{\mathbb{T}}(\theta)<\infty,  $$
	$$ \sup_{(x, \chi, \theta, j) \in \mathbb{H}(G) \times \mathbb{Z}^*} \|\mathcal{F}_c\alpha(x, \chi, \theta)(\cdot, \cdot, j)\|_{L^2(G \times \widehat{G})}<\infty $$
	and $$\sum_{j \in \mathbb{Z}^*} \|(\mathcal{F}_c \alpha(x, \chi, \theta))(\cdot,\cdot,j)\|_{L^2(G \times \widehat{G})}<\infty. $$
\end{thm}

\begin{proof} We first show the sufficiency part.   Let $f \in C_c(\mathbb{H}(G)).$ Then for all $(x,\chi, \theta) \in \mathbb{H}(G),$
	$$(T_\sigma f)(x, \chi, \theta) = \sum_{j \in \mathbb{Z}^*} tr(\rho_j^*(x, \chi, \theta) \sigma(x, \chi, \theta, j) \widehat{f}(j)).$$
Using the expression of $\sigma$ and the fact that $\widehat{f}(j)= W^j_{f^j},$ we have 
\begin{align*}
(T_\sigma f)(x, \chi, \theta) &= \sum_{j \in \mathbb{Z}^*} C_{j,G} tr(\rho_j^*(x, \chi, \theta) \rho_j(x, \chi, \theta) W^j_{\alpha(x,\chi,\theta)^{-j}}  W^j_{f^j}) \\&= \sum_{j \in \mathbb{Z}^*} C_{j,G} tr(W^j_{\alpha(x,\chi,\theta)^{-j}}  W^j_{f^j}) \\=& \sum_{j \in \mathbb{Z}^*} tr(W^j_{\alpha(x,\chi,\theta)^{-j} \times_j C_{j,G} f^j}).
\end{align*}

By \eqref{eq4.5}, we get 

\begin{align*}
(T_\sigma f)(x, \chi, \theta) &= \sum_{j \in \mathbb{Z}^*} \int_{G \times \widehat{G}}
(\alpha(x, \chi, \theta)^{-j})(x'^{-1}, \chi')\, f^j(x', \chi') \,\chi'(x'^{-1})\, d\mu_G(x')\, d\mu_{\widehat{G}}(\chi') \\&= \sum_{j \in \mathbb{Z}^*} \int_{G \times \widehat{G}}
(\alpha(x, \chi, \theta)^{-j})(x'^{-1}, \chi')\, f^j(x', \chi') \,\chi'(x'^{-1})\, d\mu_G(x')\, d\mu_{\widehat{G}}(\chi') 
\\&= \sum_{j \in \mathbb{Z}^*} \int_{G \times \widehat{G}} (\mathcal{F}_c \alpha(x, \chi, \theta))(x'^{-1}, \chi', j)\, (\mathcal{F}_c f)(x', \chi', j) \chi'(x'^{-1}) \,d\mu_G(x')\, d\mu_{\widehat{G}}(\chi') \\&= \int_{G \times \widehat{G} \times \mathbb{T}} (\alpha(x, \chi, \theta)) (x'^{-1}, \chi', \theta') \, f(x', \chi', \theta')  \chi'(x'^{-1}) d\mu_G(x')\, d\mu_{\widehat{G}}(\chi') d\mu_{\mathbb{T}}(\theta').
\end{align*}
Thus, the kernel of the operator $T_\sigma$ is a function $k$ on $(G \times \widehat{G} \times \mathbb{T}) \times (G \times \widehat{G} \times \mathbb{T})$ given as 
\begin{equation} \label{eqker}
k((x, \chi, \theta), (x', \chi', \theta'))= \chi'(x'^{-1}) \alpha(x, \chi, \theta) (x'^{-1}, \chi', \theta').
\end{equation}

Now, using Fubini theorem 
 \begin{align*}
  &\int_{\mathbb{H}(G) \times \mathbb{H}(G)} |k((x, \chi, \theta), (x', \chi', \theta'))|^2  d\mu_G(x')\, d\mu_{\widehat{G}}(\chi') d\mu_{\mathbb{T}}(\theta')  d\mu_G(x)\, d\mu_{\widehat{G}}(\chi) d\mu_{\mathbb{T}}(\theta)\\&= \int_{\mathbb{H}(G) \times \mathbb{H}(G)} |\chi'(x'^{-1}) \alpha(x, \chi, \theta) (x'^{-1}, \chi', \theta')|^2  d\mu_G(x')\, d\mu_{\widehat{G}}(\chi') d\mu_{\mathbb{T}}(\theta')  d\mu_G(x)\, d\mu_{\widehat{G}}(\chi) d\mu_{\mathbb{T}}(\theta) \\&= \int_{G\times \widehat{G} \times \mathbb{T}}  
  \|\alpha(x, \chi, \theta)\|_{L^2(\mathbb{H}(G))}^2 d\mu_G(x)\, d\mu_{\widehat{G}}(\chi)\, d\mu_{\mathbb{T}}(\theta)<\infty.
 \end{align*}
 Therefore, $T_\sigma:L^2(\mathbb{H}(G)) \rightarrow L^2(\mathbb{H}(G))$ is a Hilbert-Schmidt operator. 
 
 Conversely, assume that $T_\sigma:L^2(\mathbb{H}(G)) \rightarrow L^2(\mathbb{H}(G))$
    is a Hilbert-Schmidt operator. Then there exist a function $\alpha \in L^2(\mathbb{H}(G) \times \mathbb{H}(G))$ such that for $f \in L^2(\mathbb{H}(G))$
    $$ (T_\sigma f)(x, \chi, \theta)= \int_{\mathbb{H}(G) \times \mathbb{H}(G)} \alpha((x, \chi, \theta), (x', \chi', \theta')) f(x', \chi', \theta') d\mu_G(x')\, d\mu_{\widehat{G}}(\chi') d\mu_{\mathbb{T}}(\theta').$$
    
    Define $\alpha :\mathbb{H}(G) \rightarrow L^2(\mathbb{H}(G))$ as 
       $$ \alpha(x, \chi, \theta) (x', \chi', \theta')= \alpha((x, \chi, \theta), (x', \chi', \theta')),\,\,\,\,\, (x, \chi', \theta),\,(x', \chi', \theta') \in \mathbb{H}(G). $$
       
       Then using \eqref{eq4}, we see  
       $$\|\sigma(x, \chi, \theta, j)\|_{S_2}= \|(\mathcal{F}_c \alpha (x, \chi, \theta))(\cdot, \cdot, j)\|_{L^2(G \times \widehat{G})} \,\,\,\,\,\,(x, \chi, \theta, j) \in \mathbb{H}(G) \times \mathbb{Z}^*.  $$
       
       Then reversing the argument in the proof of sufficiency and using Theorem \ref{thm3.2}, the converse is proved.   
\end{proof}

Now we present  following corollary on a trace class pseudo-differential operator on $\mathbb{H}(G)$ and its trace formula. 
\begin{cor} Let $\alpha \in L^2(\mathbb{H}(G) \times \mathbb{H}(G)$ be such that 
	$$\int_{G \times \widehat{G} \times \mathbb{T}} |\alpha((x, \chi, \theta), (x, \chi, \theta))|\,d\mu_G(x)\, d\mu_{\widehat{G}}(\chi) d\mu_{\mathbb{T}}(\theta).$$
	Let $\sigma:\mathbb{H}(G) \times \mathbb{Z}^* \rightarrow \mathcal{B}(L^2(G)) $  be a symbol as in the Theorem \ref{thm5.1}. Then $T_\sigma:L^2(\mathbb{H}(G)) \rightarrow L^2(\mathbb{H}(G))$ is a trace class operator and  
	$$tr(T_\sigma) = \int_{\mathbb{H}(G) } \overline{\chi(x)} \alpha((x, \chi, \theta), (x^{-1}, \chi, \theta))\,\,d\mu_G(x)\, d\mu_{\widehat{G}}(\chi) d\mu_{\mathbb{T}}(\theta). $$
\end{cor}
\begin{proof}
	The proof of the corollary follows immediately from the formula \ref{eqker} on the kernel of pseudo-differential operator in Theorem \ref{thm5.1}.  
\end{proof}
\section{Trace class pseudo-differential operators on the abstract Heisenberg groups }\label{TrOp}

\begin{thm}
	Let $\sigma:\mathbb{H}(G) \times \rightarrow S_2$ be a symbol satisfying the hypothesis of Theorem \ref{thm3.2}. Then the psuedo-differential operator $T_\sigma:L^2(\mathbb{H}(G)) \rightarrow L^2(\mathbb{H}(G))$ is a trace class operator if and only if 
	$$\sigma(x, \chi, \theta, j)= C_{j,G} \, \rho_j(x, \chi, \theta) W^j_{\alpha(x, \chi, \theta)^{-j}}\,\,\,\,\, (x, \chi, \theta) \in \mathbb{H}(G),\, j \in \mathbb{Z}^*, $$ where  $\alpha:\mathbb{H}(G) \rightarrow L^2(\mathbb{H}(G))$ is a mapping such that the conditions of Theorem \ref{thm5.1} are satisfied and 
	$$\alpha(x, \chi, \theta)(x', \chi', \theta')= \int_{G\times \widehat{G} \times \mathbb{T}} \alpha_1(x, \chi, \theta)(x'', \chi'', \theta'')\alpha_2(x'', \chi'', \theta'')(x', \chi', \theta')\, d\mu_G(x'') d\mu_{\widehat{G}}(\chi'') d\mu_{\mathbb{T}}(\theta'') $$
	for all $ (x, \chi, \theta),\, (x', \chi', \theta') \in  \mathbb{H}(G);$ here $\alpha_1: \mathbb{H}(G) \rightarrow L^2(\mathbb{H}(G))$ and  $\alpha_2: \mathbb{H}(G) \rightarrow L^2(\mathbb{H}(G))$ are  such that
	
	$$\int_{G \times \widehat{G} \times \mathbb{T}} \|\alpha_i(x, \chi, \theta)\|_{L^2(\mathbb{H}(G))}< \infty\,\,\,\, i=1,2.$$ 
	
	Moreover, if $T_\sigma:L^2(\mathbb{H}(G)) \rightarrow L^2(\mathbb{H}(G))$ is a trace class operator then 
	\begin{align*}
	tr(T_\sigma)& = \int_{G \times \widehat{G} \times \mathbb{T}} \alpha(x, \chi, \theta)(x, \chi, \theta)\,  d\mu_G(x)\, d\mu_{\widehat{G}}(\chi)\, d\mu_{\mathbb{T}}(\theta) \\&= \int_{\mathbb{H}(G) } \int_{\mathbb{H}(G)} \alpha_1(x, \chi, \theta)(x', \chi', \theta')\,\alpha_2(x', \chi', \theta')(x, \chi, \theta)  d\mu_G(x')\, d\mu_{\widehat{G}}(\chi')\, d\mu_{\mathbb{T}}(\theta')\,d\mu_G(x)\, d\mu_{\widehat{G}}(\chi)\, d\mu_{\mathbb{T}}(\theta). 
	\end{align*}
	 
\end{thm}
\begin{proof}
	 The proof of this theorem follows from Theorem \ref{thm5.1} and the fact the every trace class operator is the product of two Hilbert-Schmidt operators.
\end{proof}

\section*{Acknowledgment}

Vishvesh Kumar wants to thank Prof. V. Muruganandam for his support and encouragement. He also thanks the National Institute of Science Education and
Research, the Department of Atomic Energy, Government of India, for providing excellent research facility.


\begin{thebibliography}{100}
	
	\normalsize
	\baselineskip=17pt
	
	\bibitem{Ra} S. Arati and R. Radha, Frames and Riesz bases for shift invariant spaces on the abstract Heisenberg group., \emph{ Indagationes Mathematicae}, 30(1) (2019) 106-127.
\bibitem{Dasgupta} A. Dasgupta and M. W. Wong, Hilbert-Schmidt and trace class pseudo-differential operators on Heisenberg group., \emph{J. Pseudo-Diffr. Oper. Appl.} 4 (2013) 345-359.
\bibitem{Dg }A. Dasgupta, M.W. Wong, Weyl transforms for H-type groups. J. Pseudo-Differ. Oper. Appl. 6, 11–19 (2015)
\bibitem{DW} A. Dasgupta and M. W. Wong,  Pseudo-Differential Operators on the Affine Group., \emph{ Pseudo-Differential Operators: Groups, Geometry and Applications. Trends in Mathematics. Birkhäuser, Cham}, 2017,  1-14.
\bibitem{Fo1} G. B. Folland, Harmonic Analysis in Phase Space, \emph{Princeton University Press}, Princeton, New Jersey 1989. 
\bibitem{Fo2} G. B. Folland, A course in abstract Harmonic Analysis, \emph{Studies in Advanced Mathematics,} CRC Press, Boca Raton, Florida, 1995.
\bibitem{Fis} V. Fischer and M. Ruzhansky, Quantization on nilpotent Lie groups., \emph{Progress in Mathematics}, Vol. 314, Birkhauser, 2016. xiii+557pp.
\bibitem{Man} M. Mantoiu, M. Ruzhansky, Pseudo-differential operators, Wigner transform and Weyl systems on type I locally compact groups,\emph{ Doc. Math.,} 22 (2017), 1539-1592.
\bibitem{Mol} S. Molahajloo and K. L. Wong, Pseudo-differential operators on finite abelian groups.,
\emph{J. Pseudo-Differ. Oper. Appl.} 6 (2015), 1–9.
\bibitem{Mol1} S. Molahajloo and M. W. Wong, Pseudo-differential operators on  $\mathbb{S}^{1}.$., \emph{ New Develop-
ments in Pseudo-Differential Operators, Operator Theory: Advances and Applications} 189, 2009, 297–306.



\bibitem{Resi}	M. Reed and B. Simon, Methods of Modern Mathematical Physics I., \emph{Functional Analysis, Revised and Enlarged Edition,} Academic Press, 1980.
	
	\bibitem{Kumar} Vishvesh Kumar, Pseudo-differential operators on homogeneous spaces of compact and Hausdorff groups., \emph{Forum Mathematicum}, (2018). .
	\bibitem{VW} Vishvesh Kumar and M. W. Wong, $C^*$-algebras, $H^*$-algebras and trace ideals of pseudo-differential operators on locally compact, Hausdorff and abelian groups., \emph{J. Pseudo-Diffr. Oper. Appl.} (2019). 
	
	\bibitem{Shr} R. Radha and N. Shravan Kumar, Groups, Weyl transform and Weyl multipliers associated with  locally compact abelian groups., \emph{J. Pseudo-Diffr. Oper. Appl.} 9(2) (2018) 229-245.
	\bibitem{RuV} M. Ruzhansky and V. Turunen, Pseudo-differential Operators and Symmetries: Background Analysis and Advanced Topics., \emph{Birkhauser}, Basel, 2010. 724pp. 
	\bibitem{Wei} A. Weil, Sur certains groupes d'oprateurs unitaires., \emph{Acta. Math.} 111 (1964) 143-211.
	\bibitem{Wong} M. W. Wong,  An Introduction to Pseudo-Differential Operators., Third Edition, \emph{World Scientific}, 2014.
\bibitem{Wo}	M. W. Wong, Discrete Fourier Analysis, \emph{Birkh\"{a}user}, Basel, 2011.
	
\end{thebibliography}
\end{document}